\documentclass[10pt]{amsart}

\usepackage{amsmath, amscd, amssymb}
\usepackage[frame,cmtip,arrow,matrix,line,graph,curve]{xy}
\usepackage{graphpap, color}
\usepackage[mathscr]{eucal}
\usepackage{cancel}
\usepackage{verbatim}

\numberwithin{equation}{section}

\newcommand{\CC}{\mathbb{C}}

\newcommand{\PP}{\mathbb{P}}

\newcommand{\ZZ}{\mathbb{Z}}

\newcommand{\bP}{\mathbf{P}}

\newcommand{\cal}{\mathcal}

\def\cA{{\cal A}}
\def\cB{{\cal B}}
\def\cC{{\cal C}}

\def\cE{{\cal E}}
\def\cF{{\cal F}}

\def\cK{{\cal K}}
\def\cL{{\cal L}}

\def\cO{{\cal O}}
\def\cP{{\cal P}}

\def\and{\quad{\rm and}\quad}

 \DeclareMathOperator{\Ext}{Ext}
  \DeclareMathOperator{\Hom}{Hom}

\usepackage{mathptmx}

\newtheorem{prop}{Proposition}[section]
\newtheorem{theo}[prop]{Theorem}
\newtheorem{lemm}[prop]{Lemma}
\newtheorem{coro}[prop]{Corollary}
\newtheorem{rema}[prop]{Remark}

\newtheorem{defi}[prop]{Definition}
\newtheorem{conj}[prop]{Conjecture}

\def\beq{\begin{equation}}
\def\eeq{\end{equation}}

\def\PP{\mathbb{P}}
\def\CC{\mathbb{C}}

\def\cO{\mathcal{O}}
\def\cExt{\mathcal{E}xt}

\def\Sym{\mathrm{Sym}}
\def\Ext{\mathrm{Ext}}
\def\H{\mathrm{H}}

\def\coker{\mathrm{coker} }

\def\Bl{\mathrm{Bl}}
\def\ad{\mathrm{ad}}
\def\det{\mathrm{det}}

\def\ch{\mathrm{ch}}
\def\td{\mathrm{td}}
\def\D{\mathrm{D}}

\title[Symmetric products and moduli of vector bundles of curves]{Symmetric products and moduli spaces of vector bundles of curves}

\author{Kyoung-Seog Lee and M. S. Narasimhan}

\address{Institute of the Mathematical Sciences of the Americas, University of Miami, 1365 Memorial Drive, Ungar 515, Coral Gables, FL 33146, USA}

\address{National Mathematics Initiative, Indian Institute of Science, Bangalore, India}
\address{Tata Institute of Fundamental Research, Bangalore Centre, India}

\begin{document}

\begin{abstract}
Let $X$ be a smooth projective curve of genus $g \geq 2$ and $M$ be the moduli space of rank 2 stable vector bundles on $X$ whose determinants are isomorphic to a fixed odd degree line bundle $L.$ There has been a lot of works studying the moduli and recently the bounded derived category of coherent sheaves on $M$ draws lots of attentions. It was proved that the derived category of $X$ can be embedded into the derived category of $M$ (cf. \cite{FK, Narasimhan1, Narasimhan2}). In this paper we prove that the derived category of the second symmetric product of $X$ can be embedded into derived category of $M$ when $X$ is non-hyperelliptic and $g \geq 16.$
\end{abstract}

\maketitle

\section{Introduction}

Let $X$ be a smooth projective curve of genus $g \geq 2$ and $M$ be the moduli space of rank 2 stable vector bundles on $X$ whose determinants are isomorphic to a fixed odd degree line bundle $L.$ The moduli space $M$ is a smooth projective Fano variety of dimension $3g-3,$ index $2$ (cf. \cite{Ramanan}). \\

There has been a lot of works studying the moduli and recently the bounded derived category of $M$ draws lots of attentions. It seems natural to expect that the derived category of $M$ will be closely related to $X.$ Let $E$ be a Poincar\'e bundle on $X \times M.$ Then one can define the Fourier-Mukai transform $\Phi_E : D(X) \to D(M)$ with kernel $E$ and a natural question is whether the functor $\Phi_E$ is fully-faithful or not. Indeed, it was proved that the derived category of $X$ can be embedded into the derived category of $M$ via $\Phi_E$ (cf. \cite{FK, Narasimhan1, Narasimhan2}). See \cite{BM, LM} for investigations of similar questions about moduli spaces of higher rank vector bundles on $X.$ Then a next task is to understand full semiorthogonal decomposition of $D(M).$ The second named author conjectured that the derived category of $M$ will have the following semiorthogonal decomposition. We were informed that Belmans, Galkin and Mukhopadhyay stated the same conjecture independently. See \cite{BGM, Lee} for more details.

\begin{conj}
The derived category of coherent sheaves on $M$ has the following semiorthogonal decomposition
$$ \D(M) = \langle \D(pt), \D(pt), \D(X), \D(X), \cdots, \D(X_k), \D(X_k), \cdots, \D(X_{g-1}) \rangle, $$
i.e., two copies of $\D(X_k)$ for $0 \leq k \leq g-2$ and one copy of $\D(X_{g-1}).$ Here $X_k$ means the $k$-th symmetric product of $X.$
\end{conj}

It turns out that the motive of $M$ has a motivic decomposition which is compatible with the above conjectural semiorthogonal decomposition. See \cite{GL, Lee} for precise statement and more details. Therefore a natural question is whether derived categories of symmetric products of $X$ can be embedded into the derived category of $M.$ In this paper we prove the following result.

\begin{theo}
If $X$ be a non-hyperelliptic curve with genus $g \geq 16,$ then $\D(X_2)$ can be embedded into $\D(M).$
\end{theo}

In order to prove the above theorem, we construct a vector bundle $F$ on $X_2 \times M$ which becomes a Fourier-Mukai kernel. On the way of proving the above theorem, we can see that $F_{x,y}$ and $F_{z,w}$ are distinct vector bundles on $M$ when $(x,y)$ and $(z,w)$ are distinct points in $X_2.$

\begin{prop}
There exist a vector bundle $F$ on $X_2 \times M$ such that $X_2$ is a parameter space of a family of vector bundle $F_{x,y}$ on $M$ where $(x,y) \in X_2.$
\end{prop}

We also need to compute various cohomology groups in order to prove the embedding.

\begin{theo}
Let $X$ be a non-hyperelliptic curve with genus $g \geq 16.$ The we have the following. \\
(1) If one of the four points $x,y,z,w \in X$ is different from all the others, then we have
$$ \H^i(M, E_x \otimes E_y \otimes E^*_z \otimes E^*_w)=0 $$
for every $i \in \ZZ.$ \\
(2) For two distinct points $x,z \in X$ we have
\begin{displaymath}
\H^i(M, E_x \otimes E_x \otimes E^*_z \otimes E^*_z)=\left \{ {\begin{array}{ll} 
\CC & \textrm{if $i=0$,} \\ 
\CC^2 & \textrm{if $i=1$,} \\ 
\CC & \textrm{if $i=2$,} \\ 
0 & \textrm{if $i \geq 3.$} 
\end{array}}
\right.
\end{displaymath}
(3) For $x \in X$ we have 
\begin{displaymath}
\H^i(M, E_x \otimes E_x \otimes E^*_x \otimes E^*_x)=\left \{ {\begin{array}{ll} 
\CC^2 & \textrm{if $i=0$,} \\ 
\CC^3 & \textrm{if $i=1$,} \\ 
\CC & \textrm{if $i=2$,} \\ 
0 & \textrm{if $i \geq 3.$} 
\end{array}}
\right.
\end{displaymath}
\end{theo}

From the above computation and construction of $F,$ we can show the Fourier-Mukai transform with kernel $F$ gives us the desired embedding hence obtain Theorem 1.2. 

\begin{rema}
In the above Theorem, we assume that $X$ to be a non-hyperelliptic curve with genus $g \geq 16$ due to some technical reasons. However we do not need these conditions when we construct the vector bundle $F$ and the cohomology computation above is valid for many $i \in \ZZ$ for any curve $X$ with $g \geq 3.$ We conjecture that the cohomology computation above will be the same and the vector bundle $F$ will induce an embedding $\Phi_F : D(X_2) \to D(M)$ for any smooth projective curve $X$ with $g \geq 3.$
\end{rema}

Moreover, we can prove that the embedded copy of $D(X_2)$ form a part of semiorthogonal decomposition of $D(M)$ which is compatible with the copies constructed in \cite{Narasimhan1, Narasimhan2}.

\begin{theo}
Let $X$ be a non-hyperelliptic curve with $g \geq 16,$ there exist a semiorthogonal decomposition $\D(M)=\langle \mathcal{A}, \mathcal{B} \rangle$ whose component $\mathcal{A}$ has the following semiorthogonal decomposition. 
$$ \mathcal{A} = \langle \D(pt), \D(X), \D(X_2) \rangle $$
\end{theo}

And each component of the above decomposition does not admit a nontrivial semiorthogonal decomposition.

\begin{rema}
When $g \geq 2,$ it was proved that $\D(X)$ does not admit a nontrivial semiorthogonal decomposition in \cite{Okawa}. When $g \geq 3,$ it was proved that $\D(X_2)$ does not admit a nontrivial semiorthogonal decomposition in \cite{BGL}.
\end{rema}

\bigskip

\noindent
\textbf{Conventions}.
We will work over $\CC.$ For a variety $Y,$ we will use $\D(Y)$ to denote the bounded derived category of coherent sheaves on $Y.$ We often use the same notation to denote a sheaf on $Y$ (or a morphism from $Y$) and its restriction to an open subest of $Y$ if the meaning is clear from the context. In this paper, $X$ denotes a smooth projective curve of genus $g \geq 2$ and $X_k$ denotes the $k$-th symmetric product of $X.$ Let $\theta$ be the ample divisor on $M$ which generates the Picard group of $M.$ We will fix a normalized Poincar\'e bundle $E$ on $X \times M$ such that $\det(E_x) \cong \theta$ for each $x \in X.$

\bigskip

\noindent
\textbf{Acknowledgements}.
Part of this work was done when the first named author was a research fellow of KIAS, Young Scientist Fellow of IBS-CGP and visiting IISc. He was partially supported by IBS-R003-Y1. He thanks Gadadhar Misra for kind hospitality during his stay in IISc. He also thanks Tomas Gomez and Han-Bom Moon for helpful discussions about related projects. Last but not least, he thanks Ludmil Katzarkov and Simons Foundation for partially supporting this work via Simons Investigator Award-HMS.

\bigskip

\section{Construction of functors}

Let us consider the following diagram.

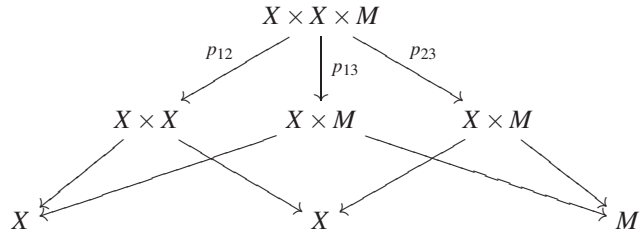
\begin{displaymath}
\xymatrix{ 
 & & X \times X \times M \ar[ld]_{p_{12}} \ar[d]^{p_{13}} \ar[rd]^{p_{23}} & & \\
 & X \times X \ar[ld] \ar[rd] & X \times M \ar[lld] \ar[rrd] & X \times M \ar[ld] \ar[rd] & \\
X & & X & & M }
\end{displaymath}

Then we have the following exact sequence.
$$ 0 \to \widetilde{F} \to p_{13}^*E \otimes p_{23}^*E \to {\bigwedge^2 E|}_{\Delta \times M} \to 0 $$

\begin{lemm}
The above vector bundle $\tilde{F}$ on $X \times X \times M$ descents to a vector bundle $F$ on $X_2 \times M.$
\end{lemm}
\begin{proof}
First, let us consider the point $(x_1, x_2) \in X \times X \setminus \Delta.$ Then the natural $S_2$-action on $X \times X \setminus \Delta$ lifts to $\tilde{F}.$ On the diagonal, the above sequence leads the following exact sequence. 
$$ 0 \to \bigwedge^2 E|_{\Delta \times M} \otimes N^\vee_{\Delta | X \times X \times M} \to \widetilde{F}|_{\Delta \times M} \to Sym^2 E |_{\Delta \times M} \to 0 $$
The permutation group $S_2$ acts on the conormal bundle by multiplying $(-1)$ and hence the stabilizer groups of the $S_2$-action on $\widetilde{F}$ are trivial. Therefore $\tilde{F}$ descents to a vector bundle $F$ on $X_2 \times M.$
\end{proof}

In general, we can use equivariant derived categories in order to construct functors. We know that there exist a fully faithful embedding $\D(X_{k}) \subset \D([X^{k}/S_k])$ from \cite{PVdB}. Then let us consider the following diagram.

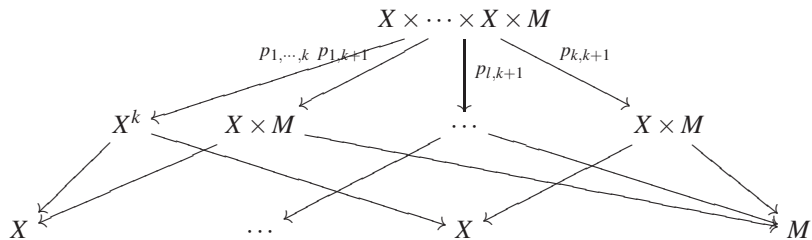
\begin{displaymath}
\xymatrix{ 
 & & & X \times \cdots \times X \times M \ar[lld]_{p_{1,\cdots,k}} \ar[ld]_{p_{1,{k+1}}} \ar[d]^{p_{l,{k+1}}} \ar[rd]^{p_{k,{k+1}}} & & \\
& X^k \ar[ld] \ar[rrd] & X \times M \ar[lld] \ar[rrrd] & \cdots \ar[ld] \ar[rrd] & X \times M \ar[ld] \ar[rd] & \\
X &  & \cdots & X & & M }
\end{displaymath}

Let us consider the bundle $\tilde{E} = p_{1,{k+1}}^*E \otimes \cdots \otimes p_{k,{k+1}}^*E$ on $X^k \times M.$ There is a natural $S_k$-action on $X^k \times M$ and the bundle $\tilde{E}$ is a $S_k$-equivariant bundle. Therefore we have a natural functor $\D([X^{k}/S_k]) \to \D(M)$ which is the descent of the Fourier-Mukai transform $\Phi_{\tilde{E}} : \D(X^{k}) \to \D(M)$ (cf. \cite{KS}). By composing these two functors we obtain the functor $\D(X_{k}) \to \D(M).$

\begin{conj}
When $1 \leq k \leq g-1,$ the functor constructed above induce a fully faithful embedding $\D(X_k) \to \D(M).$
\end{conj}

In this paper, we will prove that the above conjecture is true when $k=2,$ $X$ is a non-hyperelliptic curve with $g \geq 16.$ When $k=2,$ the functor constructed above can be explicitly described using the vector bundle $F$ on $X_2 \times M.$

\begin{prop}
When $k=2$ the functor is isomorphic to the Fourier-Mukai transform whose kernel is $F.$
\end{prop}
\begin{proof}

The equivariant derived category of $X^2$ has the following semiorthogonal decomposition (cf. \cite{CP})
$$ \D([X^2/S_2]) = \langle \D(X) \otimes \zeta, \D(X_2) \rangle $$
where the first component is given by direct image functor via diagonal embedding $\Delta : X \to X^2$ tensored by a nontrivial character $\zeta$ of $S_2$ and the second component is given by pullback via the quotient map $X^2 \to X_2.$ \\

Therefore we have the following decomposition (cf. \cite{Kuznetsov, LP}).
$$ \D([X^2/S_2] \times M) = \langle \D(X \times M) \otimes \zeta, \D(X_2 \times M) \rangle $$

Because the bundle $\tilde{E} = p_{1,3}^*E \otimes p_{2,3}^*E$ on $X^2 \times M$ has a natural linearization, $\tilde{E}$ gives an element in $\D( [X^2/S_2] \times M)$ and hence it can be projected into $\D(X_2 \times M)$ and $\D(X \times M) \otimes \zeta.$ The projections can be described via the following exact sequence. 
$$ 0 \to \widetilde{F} \to p_{13}^*E \otimes p_{23}^*E \to {\bigwedge^2 E|}_{\Delta \times M} \to 0 $$

Moreover, the Fourier-Mukai functor $\Phi_{\tilde{E}} : \D(X^2) \to \D(M)$ descents to a functor $\Phi_{\tilde{E}} : \D([X^2/S_2]) \to \D(M).$ ( Here we use the same notation to denote the descent. ) Therefore we see that the composition of the embedding $\D(X_2) \to \D([X^2/S_2])$ and the descent of $\Phi_{\tilde{E}} : \D([X^2/S_2]) \to \D(M)$ is same as the Fourier-Mukai transform whose kernel is $F.$
\end{proof}

\begin{rema} Let $F$ be the vector bundle $X_2 \times M$ constructed above and $\Phi_F$ be the Fourier-Mukai transform whose kernel is $F.$ Let $x, y \in X$ and $(x,y) \in X_2$ and $F_{x,y}=\Phi_{F}(\CC((x,y))).$ \\
(1) When $x \neq y,$ $F_{x,y} \cong E_x \otimes E_y.$ \\
(2) When $x = y,$ we have the following short exact sequence.
$$ 0 \to \bigwedge^2 E_x \to F_{x,x} \to \Sym^2 E_x \to 0 $$
\end{rema}

In the rest of the paper, we will prove that the Fourier-Mukai transform $\Phi_F$ is fully faithful when $X$ is non-hyperelliptic with $g \geq 16.$

\bigskip

\section{Geometry of projective bundles}

In this section, we prove several results which will be useful later. Let $x$ be a point of $X$ and let us assume that $L \cong \cO(x)$ in this section.

\subsection{Projective bundles}

Let $\pi : \PP E_x \to M$ be the projective bundle of $E_x.$ We can compute the (relative) dualizing sheaf of $\PP E_x$ as follows.

\begin{prop}
We have the following isomorphisms.
$$ \omega_{\PP E_x/M} \cong \pi^* \theta \otimes \cO(-2) $$
$$ \omega_{\PP E_x} \cong \pi^* \theta^{-1} \otimes \cO(-2) $$
\end{prop}
\begin{proof}
From the relative Euler sequence
$$ 0 \to \cO_{\PP E_x} \to \pi^* E_x^* \otimes \cO(1) \to T_{\PP E_x/M} \to 0 $$
we have 
$$ \omega_{\PP E_x/M} \cong \pi^* \theta \otimes \cO(-2). $$
We can compute $\omega_{\PP E_x}$ from the following formula. 
$$ \omega_{\PP E_x} \cong \omega_{\PP E_x/M} \otimes \pi^* \omega_M \cong \pi^* \theta^{-1} \otimes \cO(-2)$$
\end{proof}

\subsection{Hecke transforms}

Let us recall the Hecke transforms and constructions of \cite{Narasimhan1, NR;75}.
We have the following short exact sequence on $X \times \PP E_x.$ 
$$ 0 \to H(E) \to (1 \times \pi)^*E \to p_X^*( \CC(x) ) \otimes p_{\PP E_x}^*(\cO_{\PP E_x}(1)) \to 0 $$

Because $H(E)$ parametrizes a family of semistable vector bundles of degree 0, we have the following diagram.
\begin{displaymath}
\xymatrix{
X \times \PP E_x \ar[r]^{1 \times \varphi} \ar[d]_{1 \times \pi} & X \times M_0 \\
X \times M & }
\end{displaymath}

Taking dual of the above sequence we have the following sequence.
$$ 0 \to (1 \times \pi)^*E^* \to K(E) \to \cExt^1(p_X^*( \CC(x) ) \otimes p_{\PP E_x}^*(\cO_{\PP E_x}(1)),\cO_{X \times \PP E_x}) \to 0 $$
Here $K(E)$ is the dual of $H(E).$

\begin{rema}
Ramanan proved that there is no Poincar{\'e} bundle on any open subset of $X \times M_0$ in \cite{Ramanan}. However we can define ``adjoint bundle" $\ad G$ on $X \times M^s_0.$ See \cite[Lemma 2.2, Remark 2.3]{BV} for more details.
\end{rema}

\begin{prop}
For $y$ be a point of $X$ different from $x.$ Over $(1 \times \varphi)^{-1}(X \times M^s_0),$ we have the following isomorphism.
$$ \pi^*{\ad E_y} \cong \ad K(E)_y \cong \varphi^* \ad G_y $$
\end{prop}
\begin{proof}
For $y \neq x,$ Hecke transform does not change the bundle. Therefore we obtain the first isomorphism. The second isomorphism comes from the construction of $\ad G.$
\end{proof}

\subsection{$k$-ample}

Let us recall works of Sommese which is useful to prove vanishing of certain cohomology groups of tensor products of vector bundles (cf \cite{BM, LN}).

\begin{defi}
A line bundle $\cL$ on a smooth projective variety $Y$ is $k$-ample if (1) $\cL$ is semiample, i.e. $\cL^{\otimes m}$ is base point free for sufficiently large $m$ and (2) the fibers of the morphism $Y \to \bP(H^0(Y,\cL^{\otimes m})^*)$ have dimension $\leq k.$
\end{defi}

\begin{defi} A vector bundle $\cE$ on a projective variety $Y$ is semiample if $\cO_{\PP(\cE)}(1)$ is semiample. It is $k$-ample if $\cO_{\PP(\cE)}(1)$ is $k$-ample.
\end{defi}

Let us recall Sommese vanishing theorem. 

\begin{theo}[Sommese vanishing theorem]\cite{Sommese}
Suppose that $\cE$ is a $k$-ample bundle of rank $r.$ Then we have the following equality
$$ \H^i(Y, \omega_Y \otimes \wedge^j \cE)=0 $$
for $j>0$ and $i+j>r+k.$
\end{theo}

Let us consider the following morphism (cf. \cite{Narasimhan2}).
\begin{displaymath}
\xymatrix{
\PP E_x \ar[r]^{\varphi} \ar[d]_{\pi} & M_0 \\
M & }
\end{displaymath}

The morphism $\varphi$ is analyzed in \cite{Narasimhan2, NR;78}. The dimension of the fiber of $\varphi$ is less than or equal to $g-1.$ Therefore we see that $E_x$ is $(g-1)$-ample. From the following result, we see that $E_x \otimes E_y \otimes E_z \otimes E_w$ is $(g-1)$-ample for any four points $x,y,z,w$ of $X.$

\begin{theo}\cite{LN}
If $\cE$ is a $k$-ample vector bundle and $\cF$ is a semiample vector bundle on a compact complex manifold $Y.$ Then $\cE \otimes \cF$ is also $k$-ample.
\end{theo}

\begin{prop}
Let $x,y,z,w$ be four (not necessarily distinct) points of $X.$ Then $E_x \otimes E_y \otimes E_z^* \otimes E_w^* \otimes \omega_M^{-1}$ is a $(g-1)$-ample bundle.
\end{prop}
\begin{proof}
From the isomorphism
$$ E_x \otimes E_y \otimes E_z^* \otimes E_w^* \otimes \omega_M^{-1} \cong E_x \otimes E_y \otimes E_z \otimes E_w $$
we see that $E_x \otimes E_y \otimes E_z^* \otimes E_w^* \otimes \omega_M^{-1}$ is a $(g-1)$-ample bundle.
\end{proof}

\begin{prop}\label{highercohomology}
Let $x,y,z,w$ be four (not necessarily distinct) points of $X.$ For $i \geq g+15$ we have the following vanishing.
$$ \H^i(M,E_x \otimes E_y \otimes E_z^* \otimes E_w^*) = 0 $$
\end{prop}
\begin{proof}
From the previous Proposition and Sommese vanishing theorem we have the following vanishing
$$ \H^i(M,E_x \otimes E_y \otimes E_z^* \otimes E_w^*) \cong \H^i(M,E_x \otimes E_y \otimes E_z^* \otimes E_w^* \otimes \omega_M^{-1} \otimes \omega_M) $$
$$ \cong \H^i(M, \omega_M \otimes E_x \otimes E_y \otimes E_z \otimes E_w) = 0 $$
for $i > g-1+16-1=g+14.$
\end{proof}

\bigskip

\section{Cohomology groups}

In this section we will assume that $X$ is a non-hyperelliptic curve with genus $g \geq 16$ and compute cohomology group $\H^i(M, E_x \otimes E_y \otimes E_z^* \otimes E_w^*)$ for all $i.$ Let us explain why we need these assumptions. As in \cite{Narasimhan1}, we can try to use the morphism $\varphi : \PP E_x \to M_0$ to compute the cohomology groups. Note that the codimension of $\varphi^{-1}(\cK)$ is $g-1.$ From Sommese vanishing theorem we see that $\H^i(M,E_x \otimes E_y \otimes E_z^* \otimes E_w^*) = 0 $ for $i \geq g+15,$ and we see that analyzing the morphism $\varphi$ only over $M^s_0$ is not enough to compute $\H^i(M, E_x \otimes E_y \otimes E_z^* \otimes E_w^*)$ for all $i.$ Therefore we will use a similar strategy used in \cite{Narasimhan2}. In \cite{NR;78} the second named author and Ramanan constructed modular desingularization $\psi : H_0 \to M_0$ when $X$ is non-hyperelliptic curve of genus $g \geq 3.$ Let $\cK$ be the singular locus of $M_0$ and $\cK_0$ be the singular locus of $\cK.$ Let $Z$ be $\psi^{-1}(M_0 \setminus \cK_0)$ and $\cC$ be the conic bundle over $Z.$ Now the codimension of $\varphi^{-1}(\mathcal K_0)$ in $\PP E_x$ is $3g-2-(g-1)=2g-1.$ If $g \geq 16$ and we can compute the cohomology groups of the vector bundles over $\PP E_x \backslash \varphi^{-1}(\mathcal K_0),$ then we can compute the cohomology groups $\H^i$ for all $i$ via cohomology extension and Sommese vanishing theorem. The situation can be summarized as in the following diagram. Let $D$ be the inverse image $\psi^{-1}(\cK \setminus \cK_0)$ of $Z.$ We will use the same notation for a morphism (e.g. $\varphi$) and its restriction to an open subset of the domain.

\begin{displaymath}
\xymatrix{
\cC \ar[r]^{\widetilde{\varphi}} \ar[d]_{\tilde{\psi}} & Z \ar[d]^{\psi} \\
\PP E_x \backslash \varphi^{-1}(\mathcal K_0) \ar[r]^{\varphi} \ar[d]_{\pi} & M_0 \setminus \cK_0 \\
M & }
\end{displaymath}

\begin{prop}\cite[Remark 5.17]{NR;78}
The restriction of $T_{\pi}$ to $\varphi^{-1}(m_0)$ for $m_0 \in M_0^s$ is isomorphic to $\omega_{\varphi}.$
\end{prop}
\begin{proof}
Drezet and the second named author proved that the canonical divisor of $M_0$ is $\theta_0^{-4}$ in \cite{DN}. Dualizing sheaf of $\PP E_x$ is $\pi^* \theta^{-1} \otimes \cO(-2)$ so the relative dualizing sheaf $\omega_{\varphi}$ is $\pi^* \theta^{-1} \otimes \cO(2).$ Therefore we can check that it is isomorphic to $T_{\pi}.$
\end{proof}

We can generalize many properties of $\varphi$ to $\widetilde{\varphi}$ as follows.

\begin{prop}
The restriction of ${\widetilde{\psi}}^*T_{\pi}$ to a Hecke curve of degenerate type is $\omega_{\widetilde{\varphi}}.$
\end{prop}
\begin{proof}
Let us compare the difference between $K_{\cC} - \widetilde{\varphi}^{-1}K_Z$ and $\widetilde{\psi}^{-1} (K_{\PP E_x} - \varphi^{-1} K_{M_0}).$ We have 
$$ K_{\cC} - \widetilde{\varphi}^{-1}K_Z - \widetilde{\psi}^{-1} (K_{\PP E_x} - \varphi^{-1} K_{M_0}) = K_{\cC} - \widetilde{\psi}^{-1} K_{\PP E_x} -\widetilde{\varphi}^{-1} ( K_Z - \psi^{-1} K_{M_0}) $$

The relation between $Z$ and $M_0$ was studied in \cite{CCK, KL, Nitsure}. The map $\psi$ is an isomorphism on $M^s_0.$ Therefore the difference of the canonical divisors of $Z$ and $M_0 \setminus \cK_0$ is a multiple of the divisor $\psi^{-1}(\cK \setminus \cK_0).$ Kiem and Li computed discrepancy of the desingularization $H_0 \to M_0$ in \cite{KL} (note that over $M_0 \setminus \cK_0,$ $H_0$ is isomorphic to the desingularization constructed by Seshadri in \cite{Seshadri}) and we see that the coefficient of $\psi^{-1}(\cK \setminus \cK_0)$ in $K_Z - \psi^{-1}K_{M_0}$ is $g-2.$ From general argument (cf. \cite[Proposition IV-21]{EH}) we see that the morphism $\cC \to \PP E_x \setminus \cK_0$ is the blow-up of $\varphi^{-1}(\cK \setminus \cK_0).$ See \cite[Theorem 8.14]{NR;78} and \cite[Lemma 3.7]{Nitsure} for more details. The codimension of $\varphi^{-1}(\cK \setminus \cK_0)$ is $3g-2-(2g-1)=g-1.$ From \cite{KM} we see that the discrepancy of $\widetilde{\psi}$ is $g-2.$ Therefore we obtain the desired conclusion.
\end{proof}

Now we compute cohomology groups of $H^{i}(M,E_x \otimes E_y \otimes E_z^* \otimes E_w^*).$ We prove the following Theorem.

\begin{theo}
(1) If one of the four points $x,y,z,w \in X$ is different from all the others, then we have
$$ \H^i(M, E_x \otimes E_y \otimes E^*_z \otimes E^*_w)=0 $$
for every $i \geq 3.$ \\
(2) For two distinct points $x,z \in X$ we have
\begin{displaymath}
\H^i(M, E_x \otimes E_x \otimes E^*_z \otimes E^*_z)=\left \{ {\begin{array}{ll} 
\CC & \textrm{if $i=0$,} \\ 
\CC^2 & \textrm{if $i=1$,} \\ 
\CC & \textrm{if $i=2$,} \\ 
0 & \textrm{if $i \geq 3.$} 
\end{array}}
\right.
\end{displaymath}
(3) For $x \in X$ we have 
\begin{displaymath}
\H^i(M, E_x \otimes E_x \otimes E^*_x \otimes E^*_x)=\left \{ {\begin{array}{ll} 
\CC^2 & \textrm{if $i=0$,} \\ 
\CC^3 & \textrm{if $i=1$,} \\ 
\CC & \textrm{if $i=2$,} \\ 
0 & \textrm{if $i \geq 3.$} 
\end{array}}
\right.
\end{displaymath}
\end{theo}

Sometimes, it is convenient to assume that determinant line bundles of vector bundles which $M$ is parametrizing are isomorphic to $\cO(x).$ We need the following Lemma.

\begin{lemm}
Let $M$ (resp. $M'$) be the moduli space of stable vector bundle of rank 2 and determinant isomorphic to a fixed line bundle $L$ (resp. $L'$) of degree 1 on $X$ and $E$ (resp. $E'$) be the normalised Poincar\'{e} bundle on $X \times M$ (resp. $X \times M'$). Let $x, y, z, w$ be four (not necessarily distinct) points of $X.$ Then we have the following isomorphism.
$$ \H^i(M, E_x \otimes E_y \otimes E_z^* \otimes E_w^*) \cong \H^i(M', {E_x'} \otimes {E_y'} \otimes {E_z'}^* \otimes {E_w'}^*) $$
for every $i.$
\end{lemm}
\begin{proof}
Consider the bundle $E \otimes p_X^* \eta$ where $\eta$ is a square root of $L^{-1} \otimes L'$. From the universal property, this bundle defines an isomorphism $u : M \to M'$ such that $E \otimes p_X^* \eta \simeq (id \times u)^*E'$ because both $E$ and $E'$ are normalised Poincar\'{e} bundles. Therefore there is an isomorphism $$ E_x \otimes E_y \otimes E_z^* \otimes E_w^* \simeq u^*({E_x'} \otimes {E_y'} \otimes {E_z'}^* \otimes {E_w'}^*) $$
and from this isomorphism we have $$ \H^i(M, E_x \otimes E_y \otimes E_z^* \otimes E_w^*) \cong \H^i(M', {E_x'} \otimes {E_y'} \otimes {E_z'}^* \otimes {E_w'}^*). $$
\end{proof}

From now on we assume that $M$ is the moduli space of rank 2 stable vector bundles on $X$ whose determinants are isomorphic to $\cO(x).$

\bigskip

\subsection{Cohomology groups $H^{i}(M,E_x \otimes E_y \otimes E_z^* \otimes E_w^*)$ and its consequnces}

Let $x$ be a point in $X$ which is different from $y,z,w \in X.$ First, we need the following Lemma. 

\begin{lemm}
We have $\widetilde{\varphi}_* \widetilde{\psi}^* \pi^*(E_y \otimes E^*_z \otimes E^*_w) = 0.$
\end{lemm}
\begin{proof}
Let $m_0$ be a point in $M_0^s.$ When we restrict $\pi^*(E_y \otimes E^*_z \otimes E^*_w)$ to $\varphi^{-1}(m_0)$ is $\cO(-1)^{\oplus 8}$ (cf. \cite[Proposition 3.4]{Narasimhan1}) hence $\varphi_*\pi^*(E_y \otimes E^*_z \otimes E^*_w)=0$ over $M^s_0.$ By \cite[Corollary 1.9]{Hartshorne}, $\widetilde{\varphi}_* \widetilde{\psi}^* \pi^*(E_y \otimes E^*_z \otimes E^*_w)$ is torsion free on $Z,$ hence it is zero.
\end{proof}

Next we need to compute $R^1 \widetilde{\varphi}_* \widetilde{\psi}^* \pi^*(E_y \otimes E^*_z \otimes E^*_w).$ Note that it is supported on $\psi^{-1}(\cK \setminus \cK_0).$ Let $\cP$ be the Poincar\'e bundle on $X \times J$ and $\pi_J : X \times J \to J$ be the projection to $J.$ In order to describe $R^1 \widetilde{\varphi}_* \widetilde{\psi}^* \pi^*(E_y \otimes E^*_z \otimes E^*_w),$ let consider the following situation. 

\begin{displaymath}
\xymatrix{ 
& X \times J \ar[ld]_{\pi_X} \ar[rd]^{\pi_J} & \\
X & & J }
\end{displaymath}

We have the following sequence.
$$ 0 \to \cP^{\otimes 2} \otimes \pi_X^*{\cO(-x)} \to \cP^{\otimes 2} \to \cP^{\otimes 2}|_{x \times J} \to 0 $$

By pushforwarding the above sequence to $J$ we have the following short exact sequence and let $\cE_+ \cong R^1 {\pi_J}_*( \cP^{\otimes 2} \otimes \pi_X^*{\cO(-x)} )$ and $\cF_+ \cong R^1 {\pi_J}_*( \cP^{\otimes 2} ).$ 
$$ 0 \to \cP_x^{\otimes 2} \to R^1 {\pi_J}_*( \cP^{\otimes 2} \otimes \pi_X^*{\cO(-x)} ) \to R^1 {\pi_J}_*( \cP^{\otimes 2} ) \to 0 $$

Similarly, we have the following sequence and let $\cE_- \cong R^1 {\pi_J}_*( \cP^{\otimes -2} \otimes \pi_X^*{\cO(-x)} )$ and $\cF_- \cong R^1 {\pi_J}_*( \cP^{\otimes -2} ).$ 
$$ 0 \to \cP_x^{\otimes -2} \to R^1 {\pi_J}_*( \cP^{\otimes -2} \otimes \pi_X^*{\cO(-x)} ) \to R^1 {\pi_J}_*( \cP^{\otimes -2} ) \to 0 $$

Let $J' = J \setminus \cK_0$ and let us consider the above bundles over $J'.$ (We will use the same notation to denote the restricted bundles over $J'.$) Then the lift of $D$ to $J'$ is isomorphic to $\PP (\cF^*_+) \times_{J'} \PP(\cF^*_-)$ and the lift of $\phi^{-1}$ to $J'$ is isomorphic to the union of $\PP (\cE^*_+)$ and $\PP (\cE^*_-).$ Let us consider the following diagram.
\begin{displaymath}
\xymatrix{ 
& X \times \PP(\cE^*_+) \ar[ld]_{\pi_X} \ar[rd]^{\pi_{\PP(\cE^*_+)}} & \\
X & & \PP(\cE^*_+) }
\end{displaymath}

Over $X \times \PP(\cE^*_+)$ we have the following sequence 
$$ 0 \to \cP \otimes \pi_{\PP(\cE^*_+)}^* \cO_{\PP(\cE^*_+)}(1) \to (1 \times \pi)^*E|_{X \times \PP(\cE^*_+)} \to \cP^{-1} \otimes \pi^*_X \cO(x) \to 0 $$
where we use abuse of notation by using $\cP$ to denote the pullback of $\cP$ from $X \times J'$ to $X \times \PP(\cE^*_+)$ and also $(1 \times \pi)^*E$ to denote the pullback of the $(1 \times \pi)^*E$ on $X \times \PP E_x$ to $X \times \PP(\cE^*_+).$ Note that $\PP(\cE^*_+)$ parametrizes stable vector bundles on $X$ whose determinants are isomorphic to $\cO(x).$ \\

Similarly, we have the following sequence 
$$ 0 \to \cP^{-1} \otimes \pi_{\PP(\cE^*_-)}^* \cO_{\PP(\cE^*_-)}(1) \to (1 \times \pi)^*E|_{X \times \PP(\cE^*_-)} \to \cP \otimes \pi^*_X \cO(x) \to 0 $$
over $X \times \PP(\cE^*_-)$ where we use abuse of notation by using $\cP$ to denote the pullback of $\cP$ from $X \times J'$ to $X \times \PP(\cE^*_-)$ and also $(1 \times \pi)^*E$ to denote the pullback of the $(1 \times \pi)^*E$ on $X \times \PP E_x$ to $X \times \PP(\cE^*_-).$ Note that $\PP(\cE^*_-)$ parametrizes stable vector bundles on $X$ whose determinants are isomorphic to $\cO(x).$ \\

The map $\Bl_{\PP (\cP_x^{\otimes -2}) } \PP(\cE^*_+) \to \PP(\cF^*_+)$ gives a $\PP^1$-fibration parametrized by $\PP(\cF^*_+).$ Similarly, the map $\Bl_{\PP (\cP_x^{\otimes 2}) } \PP(\cE^*_-) \to \PP(\cF^*_-)$ gives a $\PP^1$-fibration parametrized by $\PP(\cF^*_-).$ The union of the two fibrations is isomorphic to the lift of $\widetilde{\varphi}^{-1}(D) \to D$ to $J'.$ Over a point $j \in J,$ we can describe the geometry as follows. \\

Let $V$ be a $n+1$-dimensional vector space and $l$ be a 1-dimensional subspace in $V.$ There there is a natural projection $\pi_l : \PP V \dashrightarrow \PP (V/l)$ be the point $[l] \in \PP V$ corresponding to $l$ and we can make the projection into morphism by taking blow-up $\PP V$ at $[l] \in \PP V.$ Let $E_l$ be the exceptional divisor of the blow-up. We use abuse of notation to denote the morphism $\pi_l : Bl_{[l]} \PP V \to \PP (V/l).$ The following description is standard.

\begin{lemm}
The morphism $\pi_l : Bl_{[l]} \PP V \to \PP (V/l)$ is isomorphic to the natural projection from projective bundle $\PP( \cO \oplus \cO(1)) \to \PP (V/l).$ The projection induces an isomorphism between $E_l$ and $\PP (V/l).$
\end{lemm}
\begin{proof}
We can use standard coordinate system of $\PP V$ and definition of blow-up to obtain the claim.
\end{proof}

We can compute push forward of some sheaves using the above description.

\begin{lemm}
Let $H$ be the pullback of the hyperplane of $\PP V$ of the blow-up. Then ${R \pi_l}_* \cO(-2H)$ is isomorphic to $\cO(-1)[-1].$
\end{lemm}
\begin{proof}
From the short exact sequence
$$ 0 \to \cO \to \cO({E_l}) \to \cO_{E_l}({E_l}) \to 0 $$
we have the following exact sequence.
$$ 0 \to \cO(-2H) \to \cO(-2H+{E_l}) \to \cO_{E_l}({E_l}) \to 0 $$

Because restriction of $\cO(-2H+{E_l})$ to each fiber is $\cO(-1)$ we see that $R{\pi_l}_* \cO(-2H+{E_l})$ is $0.$ By pushing forward the second exact sequence we have ${R \pi_l}_* \cO(-2H)$ is isomorphic to $\cO(-1)[-1].$
\end{proof}

From the above consideration we have the following isomorphisms.

\begin{prop}
Let $j \oplus j^{-1} \in \cK \setminus \cK_0$ and consider the fiber $\psi^{-1}(j \oplus j^{-1}) \cong \PP^{g-2} \times \PP^{g-2}.$
For three points $y,z,w \in X$ which are distinct from $x,$ the restriction of $R^1\widetilde{\varphi}_*( \widetilde{\psi}^* \pi^*(E_y \otimes E_z^* \otimes E_w^*))$ on the fiber is isomorphic to $\cO(0,-1) \oplus \cO(-1,0).$
\end{prop}
\begin{proof}
Let us consider the following Mayer-Vietoris sequence where $l_1 \cup l_2$ is the fiber of $\widetilde{\psi}$ over a point in $\PP^{g-2} \times \PP^{g-2}$ over $j \oplus j^{-1}$ and $S$ is the restriction of $\widetilde{\psi}^* \pi^*(E_y \otimes E_z^* \otimes E_w^*)$ to $l_1 \cup l_2.$
$$ 0 \to \H^0(l_1\cup l_2, S) \to H^0(l_1,
S|_{l_1})\oplus H^0(l_2, S|_{l_2}) \to H^0(l_1\cap l_2,
S|_{l_1\cap l_2}) $$
$$ \to \H^1(l_1\cup l_2, S) \to H^1(l_1,
S|_{l_1})\oplus H^1(l_2, S|_{l_2}) \to H^1(l_1\cap l_2,
S|_{l_1\cap l_2}) \to 0 $$

The morphism $H^0(l_1,S|_{l_1})\oplus H^0(l_2, S|_{l_2}) \to H^0(l_1\cap l_2,
S|_{l_1\cap l_2})$ is as follows.
$$ \H^0( ( {j_y} \otimes \cO_{l_1}(1) \oplus {j_y}^{-1}) \otimes ( {j_z} \oplus {j_z}^{-1} \otimes \cO_{l_1}(-1)) \otimes ( {j_w} \oplus {j_w}^{-1} \otimes \cO_{l_1}(-1) ) ) $$
$$ \oplus \H^0( ( {j_y}^{-1} \otimes \cO_{l_1}(1) \oplus {j_y}) \otimes ( {j_z}^{-1} \oplus {j_z} \otimes \cO_{l_1}(-1)) \otimes ( {j_w}^{-1} \oplus {j_w} \otimes \cO_{l_1}(-1))) $$
$$ \to (j_y \oplus j_y^{-1}) \otimes (j_z \oplus j_z^{-1}) \otimes (j_w \oplus j_w^{-1}) $$

This map is surjective and the kernel is generated by $j_y \otimes j_z  \otimes j_w \oplus j_y^{-1} \otimes j_z^{-1} \otimes j_w^{-1}.$
Therefore we see that $\H^1(l_1 \cup l_2, S) \cong \CC^2$ and isomorphic to $\H^1(l_1,S|_{l_1})\oplus H^1(l_2, S|_{l_2}).$ It is isomorphic to the following cohomology group.
$$ \H^1(( {j_y} \otimes \cO_{l_1}(1) \oplus {j_y}^{-1}) \otimes ( {j_z} \oplus {j_z}^{-1} \otimes \cO_{l_1}(-1)) \otimes ( {j_w} \oplus {j_w}^{-1} \otimes \cO_{l_1}(-1))) $$
$$ \oplus \H^1( ( {j_y}^{-1} \otimes \cO_{l_1}(1) \oplus {j_y}) \otimes ( {j_z}^{-1} \oplus {j_z} \otimes \cO_{l_1}(-1)) \otimes ( {j_w}^{-1} \oplus {j_w} \otimes \cO_{l_1}(-1))) $$

Therefore we have a canonical isomorphism $\H^1(l_1 \cup l_2, S) \cong j_y \otimes j_z^{-1}  \otimes j_w^{-1} \oplus j_y^{-1} \otimes j_z \otimes j_w.$ \\

Let us discuss about the morphism $\widetilde{\varphi}.$ See \cite{Narasimhan2} for more details. Let $z \in \psi^{-1} (\cK \setminus \cK_0)$ such that $\psi(z) = j \oplus j^{-1}$ and then $z$ corresponds to a point $(a,b)$ of $\PP H^1(X,j^{\otimes 2}) \times \PP H^1(X,j^{\otimes -2}).$ Then $\widetilde{\varphi}^{-1}(z)$ is a pair of lines one in $\PP_{j}=\PP \H^1(X, j^{\otimes 2} \otimes \cO(-x))$ and another in $\PP_{j^{-1}}=\PP \H^1(X, j^{\otimes -2} \otimes \cO(-x)).$ And all the lines are passing through $\PP_{j} \cap \PP_{j^{-1}}.$ Let us fix $b$ and vary $a  \in \PP \H^1(X, j^{\otimes 2}).$ Then the lines parametrized by $\PP \H^1(X,j^{\otimes 2})$ are proper transforms of lines in $\PP_j$ passing through the point $\PP_j \cap \PP_{j^{-1}}.$ Applying the previous Lemma to the family of degenerating conics we obtain the desired conclusion. 
\end{proof}

\begin{coro}
We have the following vanishing
$$ R \psi_* R^1\widetilde{\varphi}_*( \widetilde{\psi}^* \pi^*(E_y \otimes E_z^* \otimes E_w^*)) = 0 $$
\end{coro}
\begin{proof}
From the above description of $R^1\widetilde{\varphi}_*(\pi^*(E_y \otimes E_z^* \otimes E_w^*))$ we see that its derived pushforward to $\cK \setminus \cK_0$ is $0.$ Hence we obtain the desired result.
\end{proof}

\begin{prop}\label{distinct}
If one of the four points $x,y,z,w \in X$ is different from all the others, then we have
$$ \H^i(M, E_x \otimes E_y \otimes E^*_z \otimes E^*_w)=0 $$
for every $i.$
\end{prop}
\begin{proof}
Let $x$ be a point in $X$ which is different from $y,z,w \in X.$ We have the following isomorphisms for $i \leq 2g-3.$
$$ \H^i(M, E_x \otimes E_y \otimes E^*_z \otimes E^*_w) \cong \H^i(\PP E_x, \cO(1) \otimes \pi^*(E_y \otimes E^*_z \otimes E^*_w)) \cong \H^i(\PP E_x, \varphi^* \theta_0 \otimes \pi^*(E_y \otimes E^*_z \otimes E^*_w)) $$
$$ \cong \H^i(\PP E_x \setminus \varphi^{-1} \cK_0, \varphi^* \theta_0 \otimes \pi^*(E_y \otimes E^*_z \otimes E^*_w)) \cong \H^i(\cC, \widetilde{\psi}^* \varphi^* \theta_0 \otimes \widetilde{\psi}^* \pi^*(E_y \otimes E^*_z \otimes E^*_w)) $$
$$ \cong \H^i(\cC, \widetilde{\varphi}^* \psi^* \theta_0 \otimes \widetilde{\psi}^* \pi^*(E_y \otimes E^*_z \otimes E^*_w)) \cong \H^{i-1}(Z, \psi^* \theta_0 \otimes R^1 \widetilde{\varphi}_* \widetilde{\psi}^* \pi^*(E_y \otimes E^*_z \otimes E^*_w)) $$

The last isomorphism comes from the computation $\widetilde{\varphi}_* \widetilde{\psi}^* \pi^*(E_y \otimes E^*_z \otimes E^*_w)=0.$ From the above Corollary and Leray spectral sequence, we obtain the desired vanishing for $i \leq 2g-3.$ From Proposition \ref{highercohomology} and from the inclusion $\H^{2g-2}(M,E_x \otimes E_y \otimes E_z^* \otimes E_w^*) \to \H^{2g-2}(\PP E_x \setminus \varphi^{-1}(\cK_0), \cO(1) \otimes \pi^*(E_y \otimes E^*_z \otimes E^*_w)),$ we have desired result.
\end{proof}

\bigskip

\subsection{Computation of $H^{i}(M,E_x \otimes E_x \otimes E_z^* \otimes E_z^*)$ and its consequences}

In order to compute cohomology groups $H^{i}(M,E_x \otimes E_x \otimes E_z^* \otimes E_z^*),$ we need to study coherent sheaf $\widetilde{\varphi}_* \widetilde{\psi}^* \pi^* \ad E_z$ and its dual.

\begin{lemm}
The coherent sheaf $\widetilde{\varphi}_* \widetilde{\psi}^* \pi^* \ad E_z$ is a rank 3 vector bundle.
\end{lemm}
\begin{proof}
When $m \in M^s_0,$ the inverse image of $\widetilde{\varphi}$ is a smooth conic and the restriction of $\widetilde{\psi}^* \pi^* \ad E_z$ over the conic is isomorphic to $\cO^{\oplus 3}.$ Therefore $\widetilde{\varphi}_* \widetilde{\psi}^* \pi^* \ad E_z$ is a rank 3 vector bundle on $\psi^{-1}(M^s_0).$ \\

Now let us analyze $\widetilde{\varphi}_* \widetilde{\psi}^* \pi^* \ad E_z$ over $\psi^{-1}(\cK \setminus \cK_0).$ Let us consider the following Mayer-Vietoris sequence where $l_1 \cup l_2$ is the fiber of $\widetilde{\psi}$ over a point in $\PP^{g-2} \times \PP^{g-2}$ over $j \oplus j^{-1}$ and $S$ is the restriction of $\widetilde{\psi}^* \pi^* \ad E_z$ to $l_1 \cup l_2.$
$$ 0 \to \H^0(l_1 \cup l_2, S) \to H^0(l_1,
S|_{l_1})\oplus H^0(l_2, S|_{l_2}) \to H^0(l_1\cap l_2,
S|_{l_1\cap l_2}) $$
$$ \to \H^1(l_1\cup l_2, S) \to H^1(l_1,
S|_{l_1})\oplus H^1(l_2, S|_{l_2}) \to H^1(l_1\cap l_2,
S|_{l_1\cap l_2}) \to 0 $$

Let us analyze the morphism $H^0(l_1,S|_{l_1}) \oplus H^0(l_2, S|_{l_2}) \to H^0(l_1\cap l_2, S|_{l_1\cap l_2}).$ It is isomorphic to $\CC^{\oplus 2} \oplus j^{\otimes 2}_z \oplus j^{\otimes -2}_z \oplus j^{\otimes 2}_z \oplus j^{\otimes -2}_z \to \CC \oplus j^{\otimes 2}_z \oplus j^{\otimes -2}_z$ and surjective. Therefore we have $\H^0(l_1 \cup l_2, S) \cong \CC \oplus j^{\otimes 2}_z \oplus j^{\otimes -2}_z$ and $\H^1(l_1 \cup l_2, S) \cong 0.$ \\

Therefore we see that $\widetilde{\varphi}_* \widetilde{\psi}^* \pi^* \ad E_z$ is a rank 3 vector bundle.
\end{proof}

\begin{prop}
Let $\pi_l : Bl_{[l]} \PP V \to \PP (V/l)$ be the morphism as before. Let $H$ be the pullback of the hyperplane of $\PP V$ of the blow-up. Then ${R \pi_l}_* \cO(H)$ is isomorphic to $\cO \oplus \cO(1).$
\end{prop}
\begin{proof}
From the short exact sequence
$$ 0 \to \cO(-E_l) \to \cO \to \cO_{E_l} \to 0 $$
we have the following exact sequence.
$$ 0 \to \cO(H-{E_l}) \to \cO(H) \to \cO_{E_l} \to 0 $$

Because $\H^0(\cO(H-{E_l})) \cong \CC^{n}$ and ${\pi_l}_* \cO(H-{E_l})$ is a line bundle on $\PP (V/l)$ we see that ${\pi_l}_* \cO(H-{E_l}) \cong \cO(1).$ By pushing forward the second exact sequence via $\pi_l$ we have the following sequence
$$ 0 \to \cO(1) \to \pi_*\cO(H) \to \cO \to 0 $$
and ${R^i \pi_l}_* \cO(H) =0$ for $i \geq 1.$ Because $\Ext^1(\cO, \cO(1))=0,$ we have ${R \pi_l}_* \cO(H) \cong \cO \oplus \cO(1).$
\end{proof}

\begin{lemm}
We have following two injective morphism 
$$ \widetilde{\varphi}^* ( \widetilde{\varphi}_* \widetilde{\psi}^* \pi^* \ad E_z ) \to \widetilde{\psi}^* \pi^* \ad E_z $$
and
$$ \widetilde{\psi}^* \pi^* \ad E_z  \to  \widetilde{\varphi}^* ( \widetilde{\varphi}_* \widetilde{\psi}^* \pi^* \ad E_z )^*. $$
\end{lemm}
\begin{proof}
From adjunction formula we have the following morphism.
$$ \widetilde{\varphi}^* ( \widetilde{\varphi}_* \widetilde{\psi}^* \pi^* \ad E_z ) \to \widetilde{\psi}^* \pi^* \ad E_z $$
It is injective since $\widetilde{\varphi}^* ( \widetilde{\varphi}_* \widetilde{\psi}^* \pi^* \ad E_z )$ is a pure sheaf and it is injective over $M^s_0.$ \\

By taking dual to the above sequence we have
$$ \widetilde{\psi}^* \pi^* \ad E_z  \to  \widetilde{\varphi}^* ( \widetilde{\varphi}_* \widetilde{\psi}^* \pi^* \ad E_z )^* $$
and see that it is also injective due to the same reason.
\end{proof}

By pushing forward the map $$ \widetilde{\psi}^* \pi^* \ad E_z  \to  \widetilde{\varphi}^* ( \widetilde{\varphi}_* \widetilde{\psi}^* \pi^* \ad E_z )^* $$ via $\widetilde{\varphi}$ we have an injective map
$$ \widetilde{\varphi}_* \widetilde{\psi}^* \pi^* \ad E_z \to ( \widetilde{\varphi}_* \widetilde{\psi}^* \pi^* \ad E_z )^*. $$ 

In order to describe its cokernel let us consider the following situation. Note that the cokernel is supported on $D$ where $D$ be the inverse image $\psi^{-1}(\cK \setminus \cK_0)$ of $Z.$ \\

From the above discussions we have the following Lemma.

\begin{lemm}
We have the following sequence.
$$ 0 \to \cP_z^{\otimes -2} \otimes \cO(1,0) \oplus \cP_z^{\otimes 2} \otimes \cO(0,1) \to \widetilde{\varphi}_* \widetilde{\psi}^* \pi^* \ad E_z|_D \to \cO \to 0 $$
\end{lemm}
\begin{proof}
When we restrict the following sequence on $X \times \PP(\cE^*_+)$ 
$$ 0 \to \cP \otimes \pi_{\PP(\cE^*_+)}^* \cO_{\PP(\cE^*_+)}(1) \to (1 \times \pi)^*E|_{X \times \PP(\cE^*_+)} \to \cP^{-1} \otimes \pi^*_X \cO(x) \to 0 $$
to $z \times \PP(\cE^*_+)$ we have the following sequence.
$$ 0 \to \cP_z \otimes \pi_{\PP(\cE^*_+)}^* \cO_{\PP(\cE^*_+)}(1) \to  {\pi}^*E_z|_{\PP(\cE^*_+)} \to \cP_z^{-1} \to 0 $$
From the above sequence have the following diagram.
\[\xymatrix{
 & 0 \ar[d] & 0 \ar[d] & 0 \ar[d] & \\
0 \ar[r] & \cP_z^{\otimes -2} \otimes \cO_{\PP(\cE^*_+)}(1) \ar[d] \ar[r] & ... \ar[d] \ar[r] & \cO \ar[d] \ar[r] & 0 \\
0 \ar[r] & \cdots \ar[d] \ar[r] & \pi^* (E_z \otimes E^*_z)|_{\PP(\cE^*_+)} \ar[d] \ar[r] & \cdots \ar[d] \ar[r] & 0 \\
0 \ar[r] & \cO \ar[r] \ar[d] & \cdots \ar[r] \ar[d] & \cP_z^{\otimes 2} \otimes \cO_{\PP(\cE^*_+)}(-1) \ar[r]  \ar[d] & 0 \\
 & 0 & 0 & 0 &
}\]
Similarly, we have the following sequence
$$ 0 \to \cP_z^{-1} \otimes \pi_{\PP(\cE^*_-)}^* \cO_{\PP(\cE^*_-)}(1) \to \pi^*E_z|_{\PP(\cE^*_-)} \to \cP_z \to 0 $$
and similar diagram. By pulling the above diagrams to $\cC$ and applying relative Mayer-Vietoris sequence we obtain the desired short exact sequence.
\end{proof}

By taking dual we have the following sequence.

\begin{lemm}
We have the following sequence.
$$ 0 \to \cO \to (\widetilde{\varphi}_* \widetilde{\psi}^* \pi^* \ad E_z)^*|_D \to \cP_z^{\otimes -2} \otimes \cO(0,-1) \oplus \cP_z^{\otimes 2} \otimes \cO(-1,0) \to 0 $$
\end{lemm}

From the above descriptions, one can check that there is a canonical morphism $\widetilde{\varphi}_* \widetilde{\psi}^* \pi^* \ad E_z|_D \to (\widetilde{\varphi}_* \widetilde{\psi}^* \pi^* \ad E_z )^*|_D.$ We can complete the sequence as follows.

\begin{lemm}
We have the following exact sequence.
$$ 0 \to \cP_z^{\otimes -2} \otimes \cO(1,0) \oplus \cP_z^{\otimes 2} \otimes \cO(0,1) \to \widetilde{\varphi}_* \widetilde{\psi}^* \pi^* \ad E_z|_D \to (\widetilde{\varphi}_* \widetilde{\psi}^* \pi^* \ad E_z )^*|_D \to \cP_z^{\otimes -2} \otimes \cO(0,-1) \oplus \cP_z^{\otimes 2} \otimes \cO(-1,0) \to 0 $$
\end{lemm}
\begin{proof}
By composing $\widetilde{\varphi}_* \widetilde{\psi}^* \pi^* \ad E_z|_D \to \cO$ and $\cO \to (\widetilde{\varphi}_* \widetilde{\psi}^* \pi^* \ad E_z)^*|_D$ we have the following diagram.
\[\xymatrix{
 \cP_z^{\otimes -2} \otimes \cO(1,0) \oplus \cP_z^{\otimes 2} \otimes \cO(0,1) \ar[d] \ar[r] & \cP_z^{\otimes -2} \otimes \cO(1,0) \oplus \cP_z^{\otimes 2} \otimes \cO(0,1) \ar[d] \ar[r] & 0 \ar[d] & \\
\widetilde{\varphi}_* \widetilde{\psi}^* \pi^* \ad E_z|_D \ar[d] \ar[r] & \widetilde{\varphi}_* \widetilde{\psi}^* \pi^* \ad E_z|_D \ar[d] \ar[r] & 0 \ar[d] & \\
\cO \ar[d] \ar[r] & (\widetilde{\varphi}_* \widetilde{\psi}^* \pi^* \ad E_z )^*|_D \ar[d] \ar[r] & \cP_z^{\otimes -2} \otimes \cO(-1) \oplus \cP_z^{\otimes 2} \otimes \cO(-1) \ar[d] & \\
 0 \ar[r] & \cP_z^{\otimes -2} \otimes \cO(0,-1) \oplus \cP_z^{\otimes 2} \otimes \cO(-1,0) \ar[r] & \cP_z^{\otimes -2} \otimes \cO(0,-1) \oplus \cP_z^{\otimes 2} \otimes \cO(-1,0) & \\
}\]

From the snake lemma, we have the desired exact sequence.
\end{proof}

From the above discussion, we can compare the cohomology groups of $\H^i(Z, \widetilde{\varphi}_* \widetilde{\psi}^* \pi^* \ad E_z)$ and $\H^i(Z,( \widetilde{\varphi}_* \widetilde{\psi}^* \pi^* \ad E_z )^\vee).$

\begin{lemm}
We have the following isomorphism 
$$ \H^i(Z, \widetilde{\varphi}_* \widetilde{\psi}^* \pi^* \ad E_z) \cong \H^i(Z,( \widetilde{\varphi}_* \widetilde{\psi}^* \pi^* \ad E_z )^\vee) $$
for all $i.$
\end{lemm}
\begin{proof}
Because the cokernel of the map $0 \to \widetilde{\varphi}_* \widetilde{\psi}^* \pi^* \ad E_z \to ( \widetilde{\varphi}_* \widetilde{\psi}^* \pi^* \ad E_z )^\vee$ is supported on $D$ so we can complete the short exact sequence as follows.
$$ 0 \to \widetilde{\varphi}_* \widetilde{\psi}^* \pi^* \ad E_z \to ( \widetilde{\varphi}_* \widetilde{\psi}^* \pi^* \ad E_z )^\vee \to \cP_z^{\otimes -2} \otimes \cO(0,-1) \oplus \cP_z^{\otimes 2} \otimes \cO(-1,0) $$
Since the pushforward of $\cP_z^{\otimes -2} \otimes \cO(0,-1) \oplus \cP_z^{\otimes 2} \otimes \cO(-1,0)$ to $\cK \setminus \cK_0$ is zero so we obtain desired isomorphism.
\end{proof}

Then we have the following computation.

\begin{lemm}
We have the following isomorphism
$$ \H^i(\cC, \widetilde{\psi}^* ( T_{\pi} \otimes \pi^* \ad E_z) ) = 0 $$
for all $i.$
\end{lemm}
\begin{proof}
We have the following isomorphisms
$$ \H^i(\cC, \widetilde{\psi}^* ( T_{\pi} \otimes \pi^* \ad E_z) ) \cong \H^i(\cC, \omega_{\widetilde{\varphi}} \otimes \widetilde{\psi}^* \pi^* \ad E_z ) \cong \H^{i+1}(Z, ( \widetilde{\varphi}_* \widetilde{\psi}^* \pi^* \ad E_z )^\vee ) $$ 
From the previous isomorphism, we have
$$ \H^{i+1}(Z, ( \widetilde{\varphi}_* \widetilde{\psi}^* \pi^* \ad E_z )^\vee ) \cong \H^{i+1}(Z, \widetilde{\varphi}_* \widetilde{\psi}^* \pi^* \ad E_z ) \cong \H^{i+1}(\cC, \widetilde{\psi}^* \pi^* \ad E_z ) \cong \H^{i+1}(M, \ad E_z) $$
and hence we obtain the desired result.
\end{proof}

Therefore we obtain the desired cohomology computation.

\begin{prop}
For two distinct points $x,z \in X$ we have
\begin{displaymath}
\H^i(M, E_x \otimes E_x \otimes E^*_z \otimes E^*_z)=\left \{ {\begin{array}{ll} 
\CC & \textrm{if $i=0$,} \\ 
\CC^2 & \textrm{if $i=1$,} \\ 
\CC & \textrm{if $i=2$,} \\ 
0 & \textrm{if $i \geq 3.$} 
\end{array}}
\right.
\end{displaymath}
\end{prop}
\begin{proof}
First, we have the following isomorphism.
$$ \H^i(M, E^*_x \otimes E^*_z \otimes E_x \otimes E_z) \cong \H^i(M, E_x \otimes E^*_x \otimes E_z \otimes E^*_z) \cong \H^i(M, E_z \otimes E^*_z) \oplus \H^i(M, \ad E_x \otimes E_z \otimes E^*_z) $$ 

From $\pi_*T_{\pi} = \ad E_x,$ we see that 
$$H^i(M, \ad E_x \otimes E_z \otimes E_z^*) \cong H^i(\PP E_x,T_{\pi} \otimes \pi^*E_z \otimes \pi^*E_z^*) \cong \H^i(\cC, \widetilde{\psi}^* (T_{\pi} \otimes \pi^*E_z \otimes \pi^*E_z^*) ) $$ 
$$ \cong \H^i(\cC, \widetilde{\psi}^* T_{\pi} ) \oplus \H^i(\cC, \widetilde{\psi}^* ( T_{\pi} \otimes \pi^* \ad E_z) ) \cong \H^i(\cC, \omega_{\widetilde{\varphi}} ) \oplus \H^i(\cC, \widetilde{\psi}^* ( T_{\pi} \otimes \pi^* \ad E_z) ) $$
for $i \leq 2g-3.$ \\

When $i=0, 1, 2$ we can compute the cohomology groups from previous Lemmas. When $3 \leq i \leq 2g-2$ we obtain desired vanishing from the inclusion $\H^i(M,E_x \otimes E_x^* \otimes E_z \otimes E_z^*) \to \H^i(\PP E_x \setminus \varphi^{-1}(\cK_0), T_{\pi} \otimes \pi^*E_z \otimes \pi^*E_z^*).$
\end{proof}

From the above computation we have the following consequences.

\begin{prop}\label{Fxz,ext}
Let $x, z$ be two distinct points in $X.$ Then $F_{x,z}$ is a simple vector bundle and $\Ext^i_{M}(F_{x,z},F_{x,z})=0$ for all $i \geq 3.$ 
\end{prop}
\begin{proof}
When $x, z$ be two distinct points in $X$ then $F_{x,z} \cong E_x \otimes E_z.$ From the above computation we see that $F_{x,z}$ is a simple vector bundle and $\Ext^i_{M}(F_{x,z},F_{x,z})=0$ for all $i \geq 3$ when $g$ is sufficiently large.
\end{proof}

\begin{prop}\label{Fxx,Fzz}
For two distinct points $(x,x), (z,z) \in X_2$ we have $\Ext^i_{M}(F_{x,x},F_{z,z})=0$ for all $i.$
\end{prop}
\begin{proof}
We have the following exact sequences.
$$ 0 \to {\bigwedge^2E_x} \to {F_{x,x}} \to \Sym^2E_x \to 0 $$
$$ 0 \to \Sym^2E^*_z \to F^*_{z,z} \to \bigwedge^2E^*_z \to 0 $$

By tensoring the two exact sequence we have the following commutative diagram.
\[\xymatrix{
 & 0 \ar[d] & 0 \ar[d] & 0 \ar[d] & & \\
 0 \ar[r]  & {\bigwedge^2E_x} \otimes \Sym^2E^*_z \ar[d] \ar[r] & {F_{x,x}} \otimes \Sym^2E^*_z \ar[d] \ar[r] & \Sym^2E_x \otimes \Sym^2E^*_z \ar[d] \ar[r] & 0 & \\
 0 \ar[r]  & {\bigwedge^2E_x} \otimes F^*_{z,z} \ar[d] \ar[r] & {F_{x,x}} \otimes F^*_{z,z} \ar[d] \ar[r] & \Sym^2E_x \otimes F^*_{z,z} \ar[d] \ar[r] & 0 & \\
0  \ar[r]  & {\bigwedge^2E_x} \otimes \bigwedge^2E^*_z \ar[d] \ar[r] & {F_{x,x}} \otimes \bigwedge^2E^*_z \ar[d] \ar[r] & \Sym^2E_x \otimes \bigwedge^2E^*_z \ar[d] \ar[r] & 0  & \\
 & 0 & 0 & 0 & \\
}\]
We have $\H^i(M,\Sym^2E_x \otimes \bigwedge^2 E_z^*) \cong \H^i(\PP E_x, \pi^* \theta^{-1} \otimes \cO(2) ) \cong \H^i(\PP E_x, T_{\pi}) \cong \H^i(M, adE_x)$ and can compute $\H^i(M, \bigwedge^2E_x \otimes \Sym^2 E_z^*)$ similarly. Moreover, from

\begin{displaymath}
\H^i(M, E_x \otimes E_x \otimes E^*_z \otimes E^*_z)=\left \{ {\begin{array}{ll} 
\CC & \textrm{if $i=0$,} \\ 
\CC^2 & \textrm{if $i=0$,} \\ 
\CC & \textrm{if $i=0$,} \\ 
0 & \textrm{if $i \geq 3.$} 
\end{array}}
\right.
\end{displaymath}

we see that
\begin{displaymath}
\H^i(M, \Sym^2 E_x \otimes \Sym^2 E^*_z)=\left \{ {\begin{array}{ll} 
\CC & \textrm{if $i=2$,} \\ 
0 & \textrm{otherwise.} 
\end{array}}
\right.
\end{displaymath}

Then by diagram chasing, we have $\H^i(M,F_{x,x} \otimes F^*_{z,z})=0$ for every $i.$
\end{proof}

\bigskip

\subsection{Computation of $H^{i}(M,E_x \otimes E_x \otimes E_x^* \otimes E_x^*)$ and its consequnces}

Now we compute the cohomology group $H^{i}(M,E_x \otimes E_x \otimes E_x^* \otimes E_x^*).$ Let us consider the following diagram.

\begin{displaymath}
\xymatrix{
\PP E_x \ar[r]^{\varphi} \ar[d]_{\pi} & M_0 \\
M & }
\end{displaymath}

\begin{lemm}
We have the following short exact sequence.
$$ 0 \to \pi^*(E_x \otimes \theta) \otimes \cO(-1) \to \pi^* \Sym^2E_x \to \cO(2) \to 0 $$
\end{lemm}
\begin{proof}
From adjunction formula $\Hom(\Sym^2E_x, \Sym^2E_x) \cong \Hom(\pi^* \Sym^2E_x,\cO(2))$ we have a natural map $\pi^*Sym^2E_x \to \cO(2).$ This morphism is a surjection and we have an exact sequence of the following form when $A$ is an object in $\D(M).$
$$ 0 \to \pi^*A \otimes \cO(-1) \to \pi^*Sym^2E_x \to \cO(2) \to 0 $$
By twisting $\cO(-1)$ we have
$$ 0 \to \pi^*A \otimes \cO(-2) \to \pi^*Sym^2E_x \otimes \cO(-1) \to \cO(1) \to 0 $$
From our computation of $\omega_{\PP E_x / M}$ we have
$$ R^1 \pi_* \cO(-2) \cong \theta^{-1}. $$ 
Therefore we have $A \cong E_x \otimes \theta$ by applying $R\pi_*$ to the above sequence.
\end{proof}

By taking dual we have the following.

\begin{lemm}
We have the following short exact sequence.
$$ 0 \to \cO(-2) \to \pi^* \Sym^2E^*_x \to \pi^*(E^*_x \otimes \theta^{-1}) \otimes \cO(1) \to 0 $$
\end{lemm}

From the above short exact sequences we can compute the following cohomology groups.

\begin{lemm}
For $x \in X$ we have 
\begin{displaymath}
\H^i(M, \Sym^2 E_x \otimes \theta^{-1})=\left \{ {\begin{array}{ll} 
0 & \textrm{if $i=0$,} \\ 
\CC & \textrm{if $i=1$,} \\ 
0 & \textrm{if $i=2$,} \\ 
0 & \textrm{if $i \geq 3.$} 
\end{array}}
\right.
\end{displaymath}
\end{lemm}
\begin{proof}
From the following isomorphisms
$$ \H^i(M, \Sym^2 E_x \otimes \theta^{-1}) \cong \H^i(M, \ad E_x) $$
and computations in \cite{NR;75} we can compute the cohomology groups.
\end{proof}

\begin{lemm}
For $x \in X$ we have 
\begin{displaymath}
\H^i(M, \theta \otimes \Sym^2 E^*_x)=\left \{ {\begin{array}{ll} 
0 & \textrm{if $i=0$,} \\ 
\CC & \textrm{if $i=1$,} \\ 
0 & \textrm{if $i=2$,} \\ 
0 & \textrm{if $i \geq 3.$} 
\end{array}}
\right.
\end{displaymath}
\end{lemm}
\begin{proof}
From the following isomorphisms
$$ \H^i(M, \theta \otimes \Sym^2 E^*_x) \cong \H^i(M, \ad E_x) $$
and computations in \cite{NR;75} we can compute the cohomology groups.
\end{proof}

Recall that $D$ is the inverse image $\psi^{-1}(\cK \setminus \cK_0)$ of $Z.$ 

\begin{prop}
Let $\widetilde{\varphi} : \cC \to Z$ be the conic bundle and $W = R^1 \widetilde{\varphi}_* \omega_{\widetilde{\varphi}}^{\otimes 2}.$ Then $c_1(W) \cong D.$
\end{prop}
\begin{proof}
From Grothendieck-Riemann-Roch, we see that (cf. \cite[page 155]{HM}, \cite[page 302]{Mumford})
$$ \ch(\widetilde{\varphi}_! \omega_{\widetilde{\varphi}}^{\otimes 2} ) \cong \widetilde{\varphi}_* ( \ch(\omega_{\widetilde{\varphi}}^{\otimes 2}) \cdot \td(T_{\widetilde{\varphi}})) $$
$$ \cong \widetilde{\varphi}_* ( (1+c_1(\omega_{\widetilde{\varphi}}^{\otimes 2})+\frac{1}{2}c_1(\omega_{\widetilde{\varphi}}^{\otimes 2})^2 + \cdots ) \cdot ( 1 - \frac{1}{2} c_1(\omega_{\widetilde{\varphi}})+\frac{1}{12} ( c_1(\omega_{\widetilde{\varphi}})^2 + D ) + \cdots ) ) $$
$$ \cong \widetilde{\varphi}_* ( \frac{3}{2} c_1(\omega_{\widetilde{\varphi}}) + (c_1(\omega_{\widetilde{\varphi}})^2 + \frac{1}{12}( c_1(\omega_{\widetilde{\varphi}})^2 + D ) ) + \cdots  ) $$
where we use abuse of notation so that $D$ denotes the locus where the fiber of $\widetilde{\varphi}$ has nodal singularity. Because $R^i \widetilde{\varphi}_! \omega_{\widetilde{\varphi}}^{\otimes 2} = 0$ for $i \neq 1$ and $R^1 \widetilde{\varphi}_! \omega_{\widetilde{\varphi}}^{\otimes 2}=W$ is a rank 3 bundle we have
$$ c_1(W) \cong - \widetilde{\varphi}_* ( c_1(\omega_{\widetilde{\varphi}})^2 + \frac{1}{12}( c_1(\omega_{\widetilde{\varphi}})^2 + D )  ). $$ \\
From $R^1 \widetilde{\varphi}_* \omega_{\widetilde{\varphi}} \cong \cO$ and 
$$ \ch(\widetilde{\varphi}_! \omega_{\widetilde{\varphi}} ) \cong \widetilde{\varphi}_* ( \ch(\omega_{\widetilde{\varphi}} ) \cdot \td(T_{\widetilde{\varphi}})) $$
$$ \cong \widetilde{\varphi}_* ( (1+c_1(\omega_{\widetilde{\varphi}} )+\frac{1}{2}c_1(\omega_{\widetilde{\varphi}} )^2 + \cdots ) \cdot ( 1 - \frac{1}{2} c_1(\omega_{\widetilde{\varphi}})+\frac{1}{12} ( c_1(\omega_{\widetilde{\varphi}})^2 + D ) + \cdots ) ) $$
$$ \cong \widetilde{\varphi}_* ( \frac{1}{2} c_1(\omega_{\widetilde{\varphi}}) + \frac{1}{12}( c_1(\omega_{\widetilde{\varphi}})^2 + D ) + \cdots  ) $$
we see that $\widetilde{\varphi}_* ( \frac{1}{12}( c_1(\omega_{\widetilde{\varphi}})^2 + D ) )=0.$ Therefore we see that $c_1(W) \cong D.$
\end{proof}

\begin{prop}
The line subbundle of $S^2W^*$ defining conic bundle $\cC \subset \PP(W)$ is $$ \cO(-D) \subset S^2W^*. $$
\end{prop}
\begin{proof}
Let $q : S^2 W \to \cL^*$ be the quadratic form defining the conic bundle $\cC.$ The discriminant $\det ~ q$ of $q$ is a section of the line bundle $(\bigwedge^3 W^*)^{\otimes 2} \otimes {\cL^*}^{\otimes 3}$ whose 1st Chern class is $D.$ From the above Proposition, we see that $\cL^* \cong \cO(D).$
\end{proof}

From the quadratic form $q : S^2 W \to \cL^*$ we have the map $\tilde{q} : W \to W^* \otimes \cL^*.$ Then we have the following short exact sequence. 
$$ 0 \to W \to W^* \otimes \cL^* \to \coker ~ \tilde{q} \to 0 $$
Let $\widetilde{D} = \widetilde{\psi}^{-1}(D).$ We have the following exact sequence 
$$ 0 \to \cO \to \cO(\widetilde{D}) \to \cO_{\widetilde{D}}(\widetilde{D}) \to 0 $$ which is pullback of 
$$ 0 \to \cO \to \cO(D) \to \cO_D(D) \to 0 $$
since $\widetilde{\psi}$ is a flat morphism. Then we have the following.

\begin{prop}
We have the following isomorphisms. \\
(1) \begin{displaymath}
\H^i(Z,W^* \otimes \cL^*)=\left \{ {\begin{array}{ll} 
\CC & \textrm{if $i=1$,} \\ 
0 & \textrm{if $i \neq 1.$} 
\end{array}}
\right.
\end{displaymath}
(2) $$ \H^i(Z, \coker ~ \tilde{q}) = 0, ~~~~~ i \in \ZZ. $$
(3) \begin{displaymath}
\H^i(Z,W)=\left \{ {\begin{array}{ll} 
\CC & \textrm{if $i=1$,} \\ 
0 & \textrm{if $i \neq 1.$} 
\end{array}}
\right.
\end{displaymath}
\end{prop}
\begin{proof}
From relative duality we see that $W^* \cong \widetilde{\varphi}_* \omega_{\widetilde{\varphi}}^{-1}.$ 
From the following sequences of isomorphisms, $\H^i(Z,W^*) \cong \H^i(\cC, \omega_{\widetilde{\varphi}}^{-1}) \cong \H^i(\PP E_x, \omega_{\varphi}^{-1}) \cong \H^i(\PP E_x, \omega_{\pi} ) \cong \H^{i-1} (M, \cO)$ 
we have
\begin{displaymath}
\H^i(\cC, \omega_{\widetilde{\varphi}}^{-1})=\left \{ {\begin{array}{ll} 
\CC & \textrm{if $i=1$,} \\ 
0 & \textrm{if $i \neq 1.$} 
\end{array}}
\right.
\end{displaymath}

Let us consider the following short exact sequence.
$$ 0 \to \omega_{\widetilde{\varphi}}^{-1} \to \omega_{\widetilde{\varphi}}^{-1}(\widetilde{D}) \to \omega_{\widetilde{\varphi}}^{-1} \otimes \cO_{\widetilde{D}}(\widetilde{D}) \to 0 $$

Note that $\cO_{\widetilde{D}}(\widetilde{D})$ is the pull-back of $\cO_D(D)$ so it is isomorphic to $\cO(-1,-1)$ on each $\PP^{g-2} \times \PP^{g-2}$ (cf. \cite{BV, Nitsure}) Therefore we have $\widetilde{\psi}_* (\omega_{\widetilde{\varphi}}^{-1} \otimes \cO_{\widetilde{D}}(\widetilde{D})) \cong \omega_{\varphi}^{-1} \otimes \widetilde{\psi}_* (\cO_{\widetilde{D}}(\widetilde{D})) = 0.$ Therefore we have $$ \H^i(Z,W^* \otimes \cL^*) \cong \H^i(\cC, \omega_{\widetilde{\varphi}}^{-1}(\widetilde{D})) \cong \H^i(\cC, \omega_{\widetilde{\varphi}}^{-1}) $$ 
and hence obtain the desired isomorphisms. \\

For (2), we know that $\coker ~ \tilde{q} \otimes \cL$ is 2-torsion line bundle on $D.$ Then $\coker ~ \tilde{q}$ is isomorphic to 2-torsion line bundle on $D$ tensored by the normal bundle $\cO_D(D)$ and hence its restriction to $\PP^{g-2} \times \PP^{g-2}$ is isomorphic to $\cO(-1,-1).$ Therefore when we see that the pushforward $\coker ~ \tilde{q}$ to $\cK \setminus \cK_0$ is $0$ and hence obtain the desired result. \\

From the cohomology long exact sequence and (1), (2) we obtain the desired isomorphisms for (3).
\end{proof}

From the above discussions we have the following isomorphisms.

\begin{lemm}
For $x \in X$ we have 
\begin{displaymath}
\H^i(M, \Sym^2 E_x \otimes \Sym^2 E^*_x)=\left \{ {\begin{array}{ll} 
\CC & \textrm{if $i=0$,} \\ 
\CC & \textrm{if $i=1$,} \\ 
\CC & \textrm{if $i=2$,} \\ 
0 & \textrm{if $i \geq 3.$} 
\end{array}}
\right.
\end{displaymath}
\end{lemm}
\begin{proof}
By tensoring $\cO(2)$ to the following sequence
$$ 0 \to \cO(-2) \to \pi^* \Sym^2E^*_x \to \pi^*(E^*_x \otimes \theta^{-1}) \otimes \cO(1) \to 0 $$
we have
$$ 0 \to \cO \to \pi^* \Sym^2E^*_x \otimes \cO(2)  \to \pi^*(E^*_x \otimes \theta^{-1}) \otimes \cO(3) \to 0 $$

Note that $\pi^*(E^*_x \otimes \theta^{-1}) \otimes \cO(3) \cong \pi^* E^*_x \otimes \cO(1) \otimes T_{\pi}$ and consider the following short exact sequence.
$$ 0 \to T_{\pi} \to \pi^* E^*_x \otimes \cO(1) \otimes T_{\pi} \to T_{\pi}^{\otimes 2} \to 0 $$

Therefore it remains to compute $\H^i(\PP E_x, T_{\pi}^{\otimes 2}).$ We have the following isomorphisms
$$ \H^i(\PP E_x, T_{\pi}^{\otimes 2}) \cong \H^i(\cC, \widetilde{\psi}^* T_{\pi}^{\otimes 2}) \cong \H^i(\cC, \omega_{\widetilde{\varphi}}^{\otimes 2}) \cong \H^{i-1}(Z, W)  $$
and hence obtain the desired result. 
\end{proof}

\begin{prop}
For $x \in X$ we have 
\begin{displaymath}
\H^i(M, E_x \otimes E_x \otimes E^*_x \otimes E^*_x)=\left \{ {\begin{array}{ll} 
\CC^2 & \textrm{if $i=0$,} \\ 
\CC^3 & \textrm{if $i=1$,} \\ 
\CC & \textrm{if $i=2$,} \\ 
0 & \textrm{if $i \geq 3.$} 
\end{array}}
\right.
\end{displaymath}
\end{prop}
\begin{proof}
First, we have the following isomorphism.
$$ \H^i(M, E_x \otimes E_x \otimes E^*_x \otimes E^*_x) \cong \H^i(M, \cO) \oplus \H^i(M, \theta \otimes \Sym^2 E^*_x) \oplus \H^i(M, \Sym^2 E_x \otimes \theta^{-1}) \oplus \H^i(M, \Sym^2 E_x \otimes \Sym^2 E^*_x)  $$ 

Then we obtain desired isomorphism via the above discussions.
\end{proof}

From the above computation of cohomology groups, we obtain the following conclusions.

\begin{prop}\label{simple;case1}
$F_{x,x}$ is a simple vector bundle.
\end{prop}
\begin{proof}
We slightly modified the arguments used in \cite[Lemma 4.1]{NR;69} and \cite[Lemma 4.2]{CH}. Let us recall $F_{x,x}$ can be obtained from the following nontrivial extension.
$$ 0 \to \bigwedge^2E_x \to F_{x,x} \to \Sym^2E_x \to 0 $$
Suppose that $F_{x,x}$ is not simple. Then there is a nontrivial homomorphism $\alpha : F_{x,x} \to F_{x,x}$ whose rank is less than 4. If $\alpha(\theta)=0,$ then the above short exact sequence splits. Therefore we see that $\alpha(\theta) \neq 0.$ And since the composition map $\theta \to F_{x,x} \to F_{x,x} \to \Sym^2E_x$ is zero, the image of $\theta$ under $\alpha$ is $\theta.$ Similarly, we see that $\alpha$ induces a homomorphism $\alpha : \Sym^2 E_x \to \Sym^2 E_x.$ The induced homomorphism $\alpha : \Sym^2 E_x \to \Sym^2 E_x$ is not zero since the extension is nontrivial. However the only nontrivial homomorphism from $\Sym^2 E_x$ to $\Sym^2 E_x$ is an isomorphism. Therefore we see that $\alpha$ has rank 4 which gives a contradiction. Therefore we see that $F_{x,x}$ is simple.
\end{proof}

\begin{prop}\label{Fxx,ext}
For a point $(x,x) \in X_2$ we have $\Ext^i_{M}(F_{x,x},F_{x,x})=0$ for all $i \geq 3.$
\end{prop}
\begin{proof}
From the short exact sequence
$$ 0 \to \bigwedge^2E_x \to F_{x,x} \to \Sym^2E_x \to 0 $$
we obtain the following commutative diagram.
\[\xymatrix{
 & 0 \ar[d] & 0 \ar[d] & 0 \ar[d] & & \\
 0 \ar[r]  & {\bigwedge^2E_x} \otimes \Sym^2E^*_x \ar[d] \ar[r] & {F_{x,x}} \otimes \Sym^2E^*_x \ar[d] \ar[r] & \Sym^2E_x \otimes \Sym^2E^*_x \ar[d] \ar[r] & 0 & \\
 0 \ar[r]  & {\bigwedge^2E_x} \otimes F^*_{x,x} \ar[d] \ar[r] & {F_{x,x}} \otimes F^*_{x,x} \ar[d] \ar[r] & \Sym^2E_x \otimes F^*_{x,x} \ar[d] \ar[r] & 0 & \\
0  \ar[r]  & {\bigwedge^2E_x} \otimes \bigwedge^2E^*_x \ar[d] \ar[r] & {F_{x,x}} \otimes \bigwedge^2E^*_x \ar[d] \ar[r] & \Sym^2E_x \otimes \bigwedge^2E^*_x \ar[d] \ar[r] & 0  & \\
 & 0 & 0 & 0 & \\
}\]
By the above computation of cohomology groups of $\H^i(M,\cO),$ $\H^i(M,\theta \otimes \Sym^2 E_x^*),$ $\H^i(M,\theta^{-1} \otimes \Sym^2 E_x^2),$ $\H^i(M, \Sym^2 E_x \otimes \Sym^2 E_x^*)$ we have the conclusion.
\end{proof}

\bigskip

\section{Proof of the embedding of $D(X_2)$ into $D(M)$}

In this section we assume that $X$ is a non-hyperelliptic curve with genus $g \geq 16.$ In order to prove that $D(X_2)$ can be embedded into $D(M)$ via $\Phi_F,$ we check a criterion of Bondal and Orlov (cf. \cite{BO}). We need to compute various cohomology groups.

\begin{prop}
For two distinct points $(x,y), (z,w) \in X_2$ we have $\Ext^i_{M}(F_{x,y},F_{z,w})=0$ for all $i \in \ZZ$ when $g \geq 16.$
\end{prop}
\begin{proof}
Suppose that $(x,y) \neq (z,w).$ We have following cases. 
\begin{enumerate}
\item $(x,y) \not\in \Delta$ and $(z,w) \not\in \Delta$ 
\item $(x,y) \not\in \Delta$ and $(z,z) \in \Delta$
\item $(x,x) \in \Delta$ and $(z,w) \not\in \Delta$ 
\item $(x,x) \in \Delta$ and $(z,z) \in \Delta$
\end{enumerate}
For the case (1), (2), (3), we have at least one point which is different from the other three points. Therefore the claim follows from Proposition \ref{distinct}. For the case (4), the claim follows from Proposition \ref{Fxx,Fzz}.
\end{proof}

\begin{theo}
Let $\Phi_F : \D(X_2) \to \D(M)$ be the Fourier-Mukai transformation. Then $\Phi_F$ is a fully faithful embedding.
\end{theo}
\begin{proof}
From the criterion of Bondal and Orlov it is sufficient to prove the followings.
\begin{enumerate}
\item $\Hom_{M}(F_{x,y},F_{x,y})=\CC.$
\item $\Ext^i_{M}(F_{x,y},F_{x,y})=0$ for $i \geq 3.$
\item $\Ext^i_{M}(F_{x,y},F_{z,w})=0$ for all $i \in \ZZ$ if $(x,y) \neq (z,w).$
\end{enumerate}

\bigskip

The isomorphism (1) $\Hom_{M}(F_{x,y},F_{x,y})=\CC$ need to be checked for two cases; $x=y$ case and $x \neq y$ case. When $x=y,$ then the isomorphism follows from Proposition \ref{simple;case1}. When $x \neq y,$ it follows from Proposition \ref{Fxz,ext}. \\

The vanishing (2) $\Ext^i_{M}(F_{x,y},F_{x,y})=0$ for $i \geq 3$ also need to be checked for two cases; when $x=y$ and $x \neq y.$ When $x=y,$ then it follows from Proposition \ref{Fxx,ext}. When $x \neq y,$ it follows from Proposition \ref{Fxz,ext}. \\

Finally, it remains to show that (3) $\Ext^i_{M}(F_{x,y},F_{z,w})=0$ for all $i \in \ZZ$ if $(x,y) \neq (z,w)$ hold. It follows from the previous Proposition. \\

Therefore we obtain the desired result.
\end{proof}

\bigskip

\section{Semiorthogonal decomposition}

In this section we prove $D(M)$ has a semiorthogonal decomposition whose component is equivalent to $\langle \D(pt), \D(X), \D(X_2) \rangle$ when $X$ is a non-hyperelliptic curve with genus $g \geq 16.$

\begin{prop}
Let $g \geq 8.$ For any three distinct points $x, y, z \in X,$ we have the following identity.
$$ \H^i(M, E_x^* \otimes E_y^*) = 0 $$
$$ \H^i(M, E_x \otimes E_x^* \otimes E_y^*) = 0 $$
$$ \H^i(M, E_x \otimes E_y^* \otimes E_z^*) = 0 $$
for every $i \in \ZZ.$
\end{prop}
\begin{proof}
If $i \geq g+3,$ then we have
$$ \H^i(M, E_x^* \otimes E_y^*) \cong \H^i(M, \omega_M \otimes E_x \otimes E_y) = 0 $$
from Sommese vanishing theorem. From Serre duality, we have the following isomorphism.
$$ \H^i(M, E_x^* \otimes E_y^*) \cong \H^{3g-3-i}(M, E_x^* \otimes E_y^*)^* = 0 $$ when $i \leq 2g-6.$
Therefore we have 
$$ \H^i(M, E_x^* \otimes E_y^*) = 0 $$
for all $i \in \ZZ$ when $g \geq 8.$ \\

Using Sommese vanishing theorem we have
$$ \H^i(M, E_x \otimes E_x^* \otimes E_y^*) \cong \H^i(M, \omega_M \otimes E_x \otimes E_x \otimes E_y) = 0 $$
for $i \geq g+7.$ \\

For small $i,$ let us choose a point $w \in X$ which is different from $x,y.$ The proof of Proposition \ref{distinct} implies that 
$$ \H^i(M, E_x \otimes E_x^* \otimes E_y^*) \cong \H^i(\PP E_w, \pi^*( E_x \otimes E^*_x \otimes E^*_y) ) = 0 $$
for $i \leq 2g-2.$ Therefore we have
$$ \H^i(M, E_x \otimes E_x^* \otimes E_y^*) = 0 $$
for all $i$ when $g \geq 8.$ \\

Simiarly, we have $\H^i(M, E_x \otimes E_y^* \otimes E_z^*) = 0$ for every $i \in \ZZ.$
\end{proof}

\begin{prop}
Let $g \geq 8.$ For any distinct pairs $x, y \in X,$ we have the following identity.
$$ \H^i( M, F_{x,x}^* ) = 0 $$
$$ \H^i( M, F_{x,x}^* \otimes E_x ) = 0 $$
$$ \H^i( M, F_{x,x}^* \otimes E_y ) = 0 $$
for every $i \in \ZZ.$
\end{prop}
\begin{proof}
As in the previous Proposition, we have
$$ \H^i(M,E_x^* \otimes E_x^*)=0 $$
for all $i$ when $g \geq 8.$ Therefore we have $\H^i(M,\Sym^2E^*_x) = \H^i(M,\bigwedge^2E^*_x)=0$ for all $i.$ \\

From the exact sequence
$$ 0 \to \bigwedge^2E_x \to F_{x,x} \to \Sym^2E_x \to 0 $$
we have the following exact sequence.
$$ 0 \to \Sym^2E^*_x \to F^*_{x,x} \to \bigwedge^2E^*_x \to 0 $$

Therefore we have $\H^i( M, F_{x,x}^* ) = 0$ for all $i.$ \\

Using Sommese vanishing theorem we have
$$ \H^i(M, E_x \otimes E_x^* \otimes E_x^*) \cong \H^i(M, \omega_M \otimes E_x \otimes E_x \otimes E_x) = 0 $$
$$ \H^i(M, E_y \otimes E_x^* \otimes E_x^*) \cong \H^i(M, \omega_M \otimes E_y \otimes E_x \otimes E_x) = 0 $$
for $i \geq g+7.$ \\

For small $i,$ let us choose a point $w \in X$ which is different from $x,y,z.$ The proof of Proposition \ref{distinct} implies that 
$$ \H^i(M, E_x \otimes E_y^* \otimes E_z^*) \cong \H^i(\PP E_w, \pi^*( E_x \otimes E^*_y \otimes E^*_z) ) = 0 $$
for $i \leq 2g-2.$ Therefore we have
$$ \H^i(M, E_x \otimes E_x^* \otimes E_x^*) = 0 $$
and 
$$ \H^i(M, E_y \otimes E_x^* \otimes E_x^*) = 0 $$
for all $i$ when $g \geq 8.$ From these vanishing we have $\H^i( M, F_{x,x}^* \otimes E_x ) = \H^i( M, F_{x,x}^* \otimes E_y ) = 0$ for all $i \in \ZZ.$
\end{proof}

From the above discussions, we obtain the following result.

\begin{theo}
The derived category of $M$ has a semiorthogonal decomposition $\D(M) = \langle \cA, \cB \rangle$ where $\cA$ is equivalent to $ \langle \D(pt), \D(X), \D(X_2) \rangle $ when $X$ is a non-hyperelliptic curve with genus $g \geq 16.$
\end{theo}
\begin{proof}
From \cite{Narasimhan1, Narasimhan2}, we see that $\D(M)$ has a semiorthogonal decomposition having the following component.
$$ \langle \cO, \Phi_E(\D(X)) \rangle $$
Then the above propositions show that the above component lies on $\langle \Phi_F( \D(X_2)) \rangle^{\perp}.$ Therefore we obtain the desired conclusion.
\end{proof}

\bigskip

\bibliographystyle{amsplain}

\end{document}